\documentclass[11pt]{amsart}
\usepackage{amssymb,url,mathpazo}

\newtheorem{theorem}{Theorem}[section]
\newtheorem{lemma}[theorem]{Lemma}
\newtheorem{corollary}[theorem]{Corollary}

\theoremstyle{definition}
\newtheorem{definition}[theorem]{Definition}

\newtheorem{problem}[theorem]{Problem}

\theoremstyle{remark}
\newtheorem{remark}[theorem]{Remark}

\numberwithin{equation}{section}

\newcommand{\abs}[1]{\lvert#1\rvert}
\newcommand{\C}{\mathbb C}
\newcommand{\I}{\mathbb I}
\newcommand{\R}{\mathbb R}

\newcommand{\dH}{d_{\mathcal H}}
\DeclareMathOperator{\diam}{diam}
\DeclareMathOperator{\dist}{dist}
\DeclareMathOperator{\Cone}{Cone}

\begin{document}
\baselineskip5.6mm

\title[Symmetric products of the line]{Symmetric products of the line: \\ embeddings and retractions}

\author{Leonid V. Kovalev}
\address{215 Carnegie, Mathematics Department, Syracuse University, Syracuse, NY 13244-1150}
\email{lvkovale@syr.edu}
\thanks{Supported by the NSF grant DMS-0968756.}


\begin{abstract}
The $n$th symmetric product of a metric space is the set of its nonempty subsets with cardinality 
at most $n$, equipped with the Hausdorff metric. We prove that every symmetric product of the line 
is an absolute Lipschitz retract and admits a bi-Lipschitz embedding into a Euclidean space of sufficiently high dimension.  
\end{abstract}

\maketitle

\section{Introduction}

Let $X$ be a metric space. The $n$-th symmetric product of $X$ is the set of all nonempty subsets of $X$ with cardinality at most $n$. This set, denoted $X^{(n)}$, is naturally endowed with the Hausdorff metric. But it is not naturally
identified with any subset of the Cartesian product $X^n$. Indeed, Borsuk and Ulam proved that for $n\ge 4$ the symmetric product $[0,1]^{(n)}$ does not admit a topological embedding into $\R^{n}$, and asked whether
there is such an embedding into  $\R^{n+1}$~\cite{BU}. This question remains open except in low dimensions~\cite{AMS,BIY,Sch}.

In the context of metric spaces it is natural to seek embeddings that are bi-Lipschitz, not merely topological.
Borovikova and Ibragimov proved in~\cite{BI} that $\R^{(3)}$ is lipeomorphic to $\R^3$; previously these spaces were shown to be homeomorphic by Borsuk and Ulam~\cite{BU}.
Borovikova, Ibragimov and Yousefi~\cite{BIY} obtained partial results toward bi-Lipschitz embedding of $\R^{(n)}$ into some Euclidean space $\R^m$. Since there is no bi-Lipschitz counterpart of the Menger-N\"obeling theorem for topological spaces, it is not always easy to decide whether a given metric space
admits such an embedding. It turns out that $\R^{(n)}$ does.

\begin{theorem}\label{embed} The symmetric product $\R^{(n)}$ admits a bi-Lipschitz embedding into $\R^m$ where
$m = 2\, \lfloor (e-1) \, n! \rfloor$.
\end{theorem}

The dimension in Theorem~\ref{embed} is much larger than the desired embedding dimension $m=n+1$, 
which remains conjectural. On the other hand, the proof is short and easily generalizes to symmetric products of other Euclidean spaces.

\begin{theorem}\label{embed2} The symmetric product $(\R^d)^{(n)}$ admits a bi-Lipschitz embedding into $\R^m$ where $m = 2(n+1)^{d-1}\lfloor (e-1) \, n! \rfloor$.
\end{theorem}

Since $X^{(n)}$ contains an isometric copy of $X$ (namely, the set of singletons), we have the following
corollary of Theorem~\ref{embed2}.

\begin{corollary}\label{embed3} The symmetric product of a metric space $X$ admits a bi-Lipschitz
embedding into a Euclidean space if and only if $X$ does.
\end{corollary}

The proofs of the results stated above are constructive, but they do not say much about the structure of the image of the embedding. The following theorem addresses this issue.

\begin{theorem}\label{ALR} The symmetric product $\R^{(n)}$ is an absolute Lipschitz retract. 
In particular, its image under the embedding of Theorem~\ref{embed} is a Lipschitz retract of $\R^m$.
\end{theorem}

A metric space $X$ is a \emph{Lipschitz retract} of a larger space $Y$ if there is a Lipschitz map $r\colon Y\to X$
that fixes $X$ pointwise. If $X$ has this property for all choices of $Y$, it is an \emph{absolute Lipschitz retract}. 
It was previously known that $\R^{(n)}$ is quasiconvex~\cite[Theorem 4.1]{BIY}, which is a weaker property 
than being an absolute Lipschitz retract.

It remains unknown whether the property of being an absolute Lipschitz retract is inherited by symmetric products in general. The topological version of this question was raised already in~\cite{BU}. The metric version was recently considered in~\cite{Go} for spaces of unordered $n$-tuples, see also Problem~1.4 of the AIM problem list~\cite{AIMPL}.

\section{Preliminaries: metrics on cones}

A map $f\colon X\to Y$ is Lipschitz if there exists a constant $L$ such that
\[
d_Y(f(x_1),f(x_2))\le L\,d_X(x_1,x_2)\quad \text{ for all } \ x_1,x_2\in X.
\]
If $f$ satisfies a two-sided bound
\[
L^{-1}\,d_X(x_1,x_2)\le d_Y(f(x_1),f(x_2))\le L\,d_X(x_1,x_2)\quad \text{ for all } \ x_1,x_2\in X,
\]
then it is a bi-Lipschitz embedding. A surjective bi-Lipschitz embedding is called a lipeomorphism.

The Hausdorff distance $\dH(A,B)$ between two subsets $A,B$ of a metric space $X$ is the infimum of all number $r>0$
such that $A$ is contained in the $r$-neighborhood of $B$, and vice versa.

Notation  $a\lesssim b$ means that $a\le Cb$ where $C$ is either universal or depends only on dimension, such as
$n$ in $\R^n$ or $\R^{(n)}$. If both $a\lesssim b$ and $a\gtrsim b$ hold, then $a\approx b$.

The product of two metric spaces $X\times Y$ is given the Euclidean product metric,
$d^2((x_1,y_1),(x_2,y_2)) =d_X^2(x_1,x_2) + d_Y^2(y_1,y_2)$.

\begin{definition} Given a metric space $X$ of diameter at most $2$,
the  cone over $X$ is  the set
\[
\Cone(X) = X\times [0,\infty) / (X\times \{0\})
\]
with the metric
\begin{equation}\label{mymetric}
 d_c(t_1x_1, t_2x_2) = \abs{t_1-t_2}+\min(t_1,t_2)\,d(x_1,x_2).
\end{equation}
Here $tx$ is an abbreviation for $(x,t)$.
\end{definition}

To prove the triangle inequality for $d_c$, take three points $t_ix_i$, $i=1,2,3$, and let $m=\min(t_1,t_2,t_3)$. Adding the inequalities
\[
\abs{t_1-t_3}\le \abs{t_1-t_2}+\abs{t_2-t_3} - 2(\min(t_1,t_3)-m)
\]
and
\[
\min(t_1,t_3)\,d(x_1,x_3) \le m\,d(x_1,x_3) + 2(\min(t_1,t_3)-m),
\]
we arrive at
\[
\begin{split}
d_c(t_1x_1,t_3x_3)& \le \abs{t_1-t_2}+\abs{t_2-t_3} + m\,d(x_1,x_3) \\
&\le d_c(t_1x_2,t_2x_2)+d_c(t_2x_2,t_3x_3)
\end{split}
\]
as desired.

In the literature one frequently finds another cone metric
\begin{equation}\label{theirmetric}
\widetilde{d_c}(t_1x_1, t_2x_2) = \sqrt{t_1^2+t_2^2-2t_1t_2\cos d(x_1,x_2)}
\end{equation}
see, for example,~\cite[p. 91]{BBI}.
The following lemma implies the bi-Lipschitz equivalence of $\widetilde{d_c}$ and $d_c$.

\begin{lemma}\label{cone00} Let $(X,d)$ be a metric space with $\diam X\le 2$. Suppose that $\rho$ is
a metric on $\Cone(X)$ such that
\begin{equation}\label{cone001}
\begin{split}
\rho(tx_1,tx_2) &= t \,d(x_1,x_2) \\
\rho(t_1x_1,t_2x_2) & \ge  \abs{t_1-t_2} \\
\rho(t_1x,t_2x) &\le  10\abs{t_1-t_2}
\end{split}
\end{equation}
 for all $t,t_1,t_2\ge 0$ and $x,x_1,x_2\in X$.
Then $\rho\approx d_c$.
\end{lemma}

\begin{proof} Take two points $t_1x_1$ and $t_2x_2$ with $t_1\ge t_2$.
From the triangle inequality and~\eqref{cone001} it follows that
\[
\begin{split}
\rho(t_1x_1,t_2x_2) &\le \rho(t_1x_1,t_2x_1)+\rho(t_2x_1,t_2x_2)
\le 10\abs{t_1-t_2}+t_2\,d(x_1,x_2)  \\
&\le 10\,d_c(t_1x_1,t_2x_2).
\end{split}
\]
In the opposite direction, adding the inequalities
\[
\rho(t_1x_1,t_2x_2) \ge \rho(t_2x_1,t_2x_2) - \rho(t_1x_1,t_2x_1)
\ge t_2\,d(x_1,x_2)  - 10 \abs{t_1-t_2}
\]
and 
\[ 11\,\rho(t_1x_1,t_2x_2)\ge 11\abs{t_1-t_2}, \] 
we obtain $12\,\rho(t_1x_1,t_2x_2)\ge d_c(t_1x_1,t_2x_2)$. \qedhere
\end{proof}

If $X$ is bounded subset of $\R^m$, the Euclidean space structure gives yet another cone construction.
By translating and scaling $X$, we may assume that $0\in X$ and $\diam X\le 2$. Consider $\R^{m+1}$ as a linear superspace of $\R^{m}$ with the extra basis vector $e_{0}$. The set
\[
\{t x + (1-t) e_{0}\colon  t\ge 0,\ x\in X\}
\]
inherits the metric from $\R^{m+1}$ which satisfies~\eqref{cone001}.  Thus,
Lemma~\ref{cone00} yields a corollary.

\begin{corollary}\label{cone0} Suppose that a metric space $X$ with $\diam X\le 2$ admits a bi-Lipschitz embedding
into $\R^m$. Then $\Cone(X)$ admits a bi-Lipschitz embedding into $\R^{m+1}$.
\end{corollary}

The relation of cones to symmetric products is based on the following construction, which goes back to~\cite{Sch}.
Let $\I=[0,1]$ and consider the spaces
\begin{equation}\label{special}
\I_*^{(n)}=\{A\in \I^{(n)}\colon 0,1\in A\}, \qquad n\ge 2.
\end{equation}
Note that $\I_*^{(n)}$ is an $(n-2)$-dimensional space: for example, $\I_*^{(2)}=\{\{0,1\}\}$ is a singleton and
$\I_*^{(3)}=\{\{0,t,1\}\colon 0\le t\le 1\}$ is a circle.
The space $\I_*^{(4)}$ is the well-known \emph{dunce hat} and for $n>4$ the spaces
 $\I_*^{(n)}$ could be called higher-dimensional dunce hats~\cite{AMS}.

\begin{lemma}\label{cone1} For $n\ge 2$ the space $\R^{(n)}$ is lipeomorphic to $\R\times  \Cone (\I_*^{(n)})$.
\end{lemma}

\begin{proof} Let $Z=\{B\in \R^{(n)}\colon \min B = 0\}$.
Define the map $f\colon \R^{(n)}\to \R\times Z$ by sending
each set $A\in\R^{(n)}$ to $(\min A, A-\min A)$. It is evident that $f$ is Lipschitz, and so is
its inverse $(b,B)\mapsto B+b$. It remains to show that $Z$ is lipeomorphic to the cone
over $\I_*^{(n)}$.

Every set $B\in Z$ can be written as $tE$ with $E\in \I_*^{(n)}$ and $t=\max B$. 
This gives a bijection between $B$ and $\Cone(\I_*^{(n)})$.
It is easy to see that the Hausdorff metric $\dH$ on $B$ satisfies~\eqref{cone001}. Indeed,
$\dH(tE_1,tE_2)= t \,\dH(E_1,E_2)$ is trivial. To prove $\dH(t_1E_1,t_2E_2) \ge  \abs{t_1-t_2}$,
assume $t_1\ge t_2$ and observe that $\dist(t_1,t_2 E_2)=t_1-t_2$. Finally,
$\dH(t_1E,t_2E) \le \abs{t_1-t_2}$ because for every $x\in E$ the point $t_1x\in t_1E$ is within
distance $\abs{t_1-t_2}$ of the point $t_2x \in t_2E$. Thus, $Z$ is lipeomorphic to $\Cone (\I_*^{(n)})$.
\end{proof}

Combining Corollary~\ref{cone0} and Lemma~\ref{cone1} yields the following statement.

\begin{corollary}\label{cone9} If $\I_*^{(n)}$ admits a bi-Lipschitz embedding into $\R^m$, then
$\R^{(n)}$ admits a bi-Lipschitz embedding into $\R^{m+2}$.
\end{corollary}

\section{Bi-Lipschitz embeddings}

Let $\C$ denote the complex plane with the standard Euclidean metric. The following result is proved in~\cite{BIY} as the first step of the proof of Theorem~3.1. 

\begin{lemma}\label{prep1}~\cite{BIY} For $n\ge 2$ the set $\I_*^{(n)}$ admits a bi-Lipschitz embedding into $\C^{(n-1)}$.
\end{lemma}

The proof of Lemma~\ref{prep1} in~\cite{BIY} proceeds by mapping $\I$ onto the unit circle in $\C$ via $t\mapsto \exp(2\pi it)$. Since both $0$ and $1$ are mapped to the same point, this embeds $\I_*^{(n)}$ into $\C^{(n-1)}$.

The next step is to embed  $\C^{(n-1)}$ into the Cartesian product of several copies of $\R^{(n-1)}$.
Actually, we will prove a more general result.

\begin{lemma}\label{tomo} For $n,d\ge 2$ there exists a bi-Lipschitz embedding
\begin{equation}\label{tomores}
g\colon (\R^d)^{(n-1)}\to \underbrace{(\R^{d-1})^{(n-1)}\times \dots \times (\R^{d-1})^{(n-1)}}_{n \text{ times}}
\end{equation}
\end{lemma}

The embedding $g$ is obtained by projecting finite subsets of $\R^d$ onto $n$ hyperplanes in generic position. This idea is not new. R\'enyi and Haj\'os proved that every $(n-1)$-point subset of the plane is uniquely determined by its projections onto $n$ lines~\cite{Re}. This result was extended to higher dimensions by
Heppes~\cite{He}.
Subsequently, the problem of recovering finite sets from projections was extensively studied in the subject of discrete tomography~\cite{BL,BD,Gbook,GG}.
One should note, however, that in the aforementioned works projections are taken with multiplicities,
while symmetric products are multiplicity-blind.

\begin{proof} Fix $n$ distinct lines $L_1,\dots, L_n $ in $\R^d$. Let $g_j$ be the orthogonal projection
onto the orthogonal complement of $L_j$. The map $g_j$ induces a $1$-Lipschitz
map from $(\R^d)^{(n-1)}$ to $\R^{(n-1)}$, also denoted $g_j$. The product map $g=(g_1,\dots,g_{n})$ is also Lipschitz. It remains to prove the lower distance bound for $g$.

For $r>0$ let $T_j(r)$ be the open $r$-neighborhood of the line $L_j$, i.e., an open circular cylinder
of radius $r$. Since the lines $L_j$ are distinct, the intersections $T_j(r)\cap T_k(r)$, $j\ne k$, are bounded
sets. Let $M$ be a number such that
\begin{equation}\label{intersect}
T_j(r)\cap T_k(r) \subset \{x\colon \abs{x}<Mr\} \quad r>0, \ j\ne k.
\end{equation}

Consider distinct sets $A,B\in (\R^d)^{(n-1)}$ and let $\rho$ be the Hausdorff distance between them.
Suppose that the Hausdorff distance between $g_j(A)$ and $g_j(B)$ is less than $\rho/M$ for all $j$. This will lead to a contradiction, completing the proof.

Interchanging $A$ and $B$ if necessary, we may assume that there exists a point $a\in A$ such that
\begin{equation}\label{far1}
\abs{a-b}\ge \rho \quad \text{ for all } \ b\in B.
\end{equation}
On the other hand, $\dist(g_j(a),g_j(B))<\rho/M$, which implies that there exists a point $b\in B$
such that $b-a\in T_j(\rho/M)$. We have $n$ cylinders $T_j(\rho/M)$, while the cardinality of $B$ is at most $n-1$.  It follows that for some $b\in B$ the point $b-a$ lies in the intersection of two cylinders.
From~\eqref{intersect} we have $\abs{b-a}<\rho$, which contradicts~\eqref{far1}.
\end{proof}

\begin{proof}[Proof of Theorem~\ref{embed}] We proceed by induction on $n$. The base case $n=1$ is trivial,
since $\R^{(1)}$ is isometric to $\R$. Suppose that $\R^{(n-1)}$ admits a bi-Lipschitz embedding into $\R^{m}$
where
\[
m=2\,\lfloor (e-1)\,(n-1)! \rfloor = 2\, (n-1)!\sum_{k=1}^{n-1}\frac{1}{k!}.
\]
Lemmas~\ref{prep1} and~\ref{tomo} imply that $\I_*^{(n)}$ admits a bi-Lipschitz embedding into $\R^{nm}$.
By Corollary~\ref{cone9},  $\R^{(n)}$ admits a bi-Lipschitz embedding into $\R^{nm+2}$. It remains
to observe that
\[
nm+2 = 2+ 2\,n!\sum_{k=1}^{n-1}\frac{1}{k!} = 2\,n!\sum_{k=1}^{n}\frac{1}{k!} = 2\,\lfloor (e-1)\,n! \rfloor. \qedhere
\]
\end{proof}

\begin{proof}[Proof of Theorem~\ref{embed2}] Repeated application of Lemma~\ref{tomo} yields a bi-Lips\-chitz embedding of $(\R^d)^n$ into the Cartesian product of $(n+1)^{d-1}$ copies of $\R^{(n)}$. 
It remains to apply Theorem~\ref{embed}. 
\end{proof} 

\section{Lipschitz retractions}

A subset $Y$ of a metric space $X$ is a \emph{Lipschitz retract} of $X$ if there exists a Lipschitz map
$f\colon X\to Y$ that fixes $Y$ pointwise. A metric space is an \emph{absolute Lipschitz retract} if it is a Lipschitz
retract of any metric space containing it.

For any metric space $X$ and any positive integers $k<n$ we have a natural inclusion $X^{(k)}\subset X^{(n)}$. In general
$X^{(k)}$ is not a Lipschitz (or even topological) retract of $X^{(n)}$. For example, if $X$ is the circle $S^1$,
then $X^{(3)}$ is homeomorphic to $S^3$~\cite{Bo} which, being simply connected, does not retract onto $X^{(1)}=S^1$. This suggests a potentially interesting problem. 

\begin{problem} Characterize the metric spaces $X$ such that $X^{(k)}$ is a Lipschitz retract of $X^{(n)}$ whenever $k<n$.
\end{problem}

The following lemma shows that the line $\R$ and its subintervals are among such spaces. Its proof relies on the tree structure of $\R$ and does not immediately extend to $\R^d$.

\begin{lemma}\label{lr} Let $X$ be a nonempty connected subset of $\R$. Then for any integers  $1\le k< n$ there is a Lipschitz retraction $r\colon X^{(n)} \to X^{(k)}$.
\end{lemma}

\begin{proof} We may assume that $0\in X$, via translation.
It suffices to consider $k=n-1$, from which the general case follows by induction.
Given a set $A\subset X$ of cardinality $n$, order its elements $a_1<\dots<a_n$ and let
$\delta(A)=\min \{a_j-a_{j-1}\colon j=2,\dots n\}$.  For $j=1,\dots,n$ let
\[
a_j' = \begin{cases}  \min (0, a_j+(n-j)\delta) \quad &\text{if } a_j\le 0 \\
\max (0, a_j-j\delta) \quad &\text{if } a_j > 0
\end{cases}
\]
By construction, $a_1'\le \dots \le a_n'$ and at least two of these numbers are equal, i.e., the pair that realizes
the minimal distance $\delta(A)$.
We set $r(A)=\{a_1',\dots, a_n'\}$. For sets $A\subset X$ of cardinality less than $n$, define
$\delta(A)=0$ and $r(A)=A$.

To prove that $r$ is Lipschitz, we fix $A,B\in X^{(n)}$. The definitions of $\delta$ and $r$ imply
\begin{equation}\label{lr0}
\abs{\delta(A)-\delta(B)}  \le 2\,\dH(A,B)
\end{equation}
and
\begin{equation}\label{lr1}
\dH(A,r(A)) \le \max_j \abs{a_j-a_j'} \le  n\,\delta(A).
\end{equation}

\emph{Case 1}: $\delta(A)\le 2\,\dH(A,B)$. Then $\delta(B)\le 4\,\dH(A,B)$ by~\eqref{lr0}. From the triangle
inequality and~\eqref{lr1} it follows that
\[
\dH(r(A),r(B))\le \dH(A,B)+n\,\delta(A)+n\,\delta(B) \le (6n+1)\,\dH(A,B).
\]

\emph{Case 2}: $\delta(A)> 2\,\dH(A,B)$. Then $\delta(B)>0$ by~\eqref{lr0}. Order the elements of each set as
$a_1<\dots <a_n$ and $b_1<\dots < b_n$. Observe that the intervals 
\[ \left\{x\colon \abs{x-a_j} \le \dH(A,B) \right\},\quad j=1,\dots,n \]
are disjoint, and each of them contains an element of $B$. Therefore, $\abs{a_j-b_j}\le \dH(A,B)$ for all $j$. If $a_j$ and $b_j$ have opposite signs, then $\abs{a_j'-b_j'}\le \abs{a_j-b_j}$ by definition. If $a_j,b_j\ge 0$, then
\begin{equation*}
\begin{split}
\abs{a_j'-b_j'}  & \le 
\abs{(a_j-j\delta(A))- (b_j-j\delta(B))} \le \abs{a_j-b_j}+n\abs{\delta(A)-\delta(B)}  \\
& \le (2n+1)\,\dH(A,B).
\end{split}
\end{equation*} 
The case $a_j,b_j\le 0$ is treated the same way. We conclude that 
\[\dH(r(A),r(B))\le \max_{j} \abs{a_j'-b_j'}\le  (2n+1)\,\dH(A,B). \qedhere \]
\end{proof}

The following Lipschitz decomposition lemma is similar to Proposition~1.6 in~\cite{DLS}.
The subject of~\cite{DLS} is unordered $n$-tuples rather than subsets, but this has little effect on the proof.

\begin{lemma}\label{lipdecomp} Let $X$ and $Y$ be metric spaces, $\diam X=D<\infty$.
Suppose that $f\colon X\to Y^{(n)}$ is an $L$-Lipschitz function such that
\begin{equation}\label{diam}
\diam f(x_0)> 3LD(n-1) \quad \text{ for some }\ x_0\in X.
\end{equation}
Then there are $L$-Lipschitz functions $g,h\colon X\to Y^{(n-1)}$ such that $f(x)=g(x)\cup h(x)$ for all $x\in X$.
\end{lemma}

\begin{proof} Following~\cite{DLS}, we consider the family $\mathcal{S}$ of all sets $E\subset f(x_0)$ such that  $\diam E<3LD(\abs{E}-1)$. Ordered by inclusion, $\mathcal{S}$ has maximal elements. Choose and fix such a maximal set $E$, and note that $E$ is a proper
subset of $f(x_0)$. The maximality of $E$ implies that
\begin{equation}\label{dicho}
\dist(y,  E) > 3LD \quad \text{ for all } \ y\in f(x_0)\setminus E
\end{equation}
for otherwise $E\cup \{f(x_0\}$ would be in $\mathcal{S}$. Let
\[
G = \{y\in Y\colon \dist(y,E)\le LD\};  \quad H = \{y\in Y\colon \dist(y,f(x_0)\setminus E)\le LD\}.
\]
We claim that the functions $g(x)=f(x)\cap G$ and $h(x)=f(x)\cap H$ have the desired properties.

Indeed, for every $x\in X$ the set $f(x)$ is within Hausdorff distance $LD$ of $f(x_0)$. It follows that
$f(x)\subset G\cup H$ and both intersections $f(x)\cap G$ and $f(x)\cap H$ are nonempty. This implies
$g(x),h(x)\in Y^{(n-1)}$. To check the Lipschitz property, take $x_1,x_2\in X$ and let $\rho=\dH(f(x_1),f(x_2))$. Since  
$\rho \le LD < \dist(G,H)$, every point of $f(x_1)\cap G$ must be within distance $\rho$ of $f(x_2)\cap G$, and vice versa. 
Hence  
\[
\dH(g(x_1),g(x_2)) \le \rho \le L\,d_X(x_1,x_2) 
\]
and the same applies to $h$. 
\end{proof}

The following Lipschitz homotopy lemma parallels Lemma~1.8 in~\cite{DLS}. 
It  says that the metric space $\R^{(n)}$ is \emph{Lipschitz $k$-connected} for all $k=1,2,\dots$, that is, all Lipschitz homotopy groups of $\R^{(n)}$ are trivial.
Here the difference between $\R^{(n)}$ and the space of unordered $n$-tuples considered in~\cite{DLS} is more significant: it requires an appeal to the Lipschitz retraction in Lemma~\ref{lr}.

\begin{lemma}\label{homlem} For each integer $n\ge 1$ there is a constant $C_n$ such that for any closed 
ball $B\subset \R^k$ ($k=1,2,\dots$) and any $L$-Lipschitz map $f\colon \partial B\to \R^{(n)}$, there is a $C_nL$-Lipschitz  map $\widetilde{f}\colon B\to \R^{(n)}$ that agrees with $f$ on $\partial B$. 
\end{lemma}

\begin{proof} We may assume that $B$ is the unit ball $\{x\in\R^k\colon \abs{x}\le 1\}$. 
As in~\cite{DLS} we proceed by induction on $n$, the case $n=1$ being well-known. Assuming the lemma proved for all
$n'<n$, we pick $x_0\in \partial B$ and consider two cases. 

\emph{Case 1}: The assumption of Lemma~\ref{lipdecomp} is satisfied, that is, $\diam f(x_0)> 6L(n-1)$. Decompose  $f=g\,\cup\, h$ as in the lemma. By the inductive hypothesis the maps $g$ and $h$ have $C_{n-1}L$-Lipschitz extensions $\widetilde{g},\widetilde{h}\colon B\to \R^{(n-1)}$.
Their union $\widetilde{g}\,\cup\, \widetilde{h}$ is a $C_{n-1}L$-Lipschitz map of $B$ into $\R^{(2n-2)}$, and it  agrees with $f$ on $\partial B$. Applying the Lipschitz retract 
from $\R^{(2n-2)}$ onto $\R^{(n)}$, we obtain the desired Lipschitz extension $\widetilde{f}$. 

\emph{Case 2}: $\diam f(x_0) \le 6L(n-1)$. Translating $f$, may assume that $0\in f(x_0)$. For $x\in \partial B$ and $0\le r\le 1$ define $\widetilde{f}(rx) = r\,f(x)$. Clearly, $\widetilde{f}$ is $L$-Lipschitz on every 
sphere $r\,\partial B$. Also, for any fixed $x\in\partial B$ the map $r\mapsto r\,f(x)$ is $6L(n-1)$-Lipschitz, since $f(x)$ is contained in the ball of radius $6L(n-1)$ centered at the origin. It follows that $\widetilde{f}$ is Lipschitz on $B$, with a constant of the form $C_nL$.  
\end{proof} 

\begin{remark} The proof of Lemma~\ref{homlem} would immediately extend to the space $(\R^d)^{(n)}$ with $d\ge 2$ if we had a version of
Lemma~\ref{lr} for this space.  
\end{remark}

\begin{proof}[Proof of Theorem~\ref{ALR}]  Let 
$f\colon \R^{(n)}\to \R^m$ be the bi-Lipschitz embedding provided by Theorem~\ref{embed}.
Being Lipschitz $k$-connected for all $k$ (Lemma~\ref{homlem}), the set $E=f(\R^{(n)})$  enjoys the following Lipschitz  
extension property: every Lipschitz map from a subset of $\R^m$ to $E$ extends to a Lipschitz map from 
$\R^m$ to $E$ (see Corollary~1.7 in~\cite{LS} or Theorem 6.26 in~\cite{BB}). 
Extending the identity map $\mathrm{id}\colon E\to E$ in this way, we obtain the desired 
Lipschitz retraction $r\colon \R^m \to E$. Since $\R^m$ is an absolute Lipschitz retract, so are $E$
and $\R^{(n)}$. 
\end{proof}

\bibliographystyle{amsplain}

\end{document}